\documentclass[11pt]{amsart}
\usepackage{array, amsmath, enumerate, color}
\usepackage{amssymb}
\usepackage{graphicx,subfigure}

\makeatletter
 \usepackage{fullpage} 
 \usepackage{amsthm}

\usepackage{fullpage}

\makeatother

\begin{document}

\newtheorem{thm}{Theorem}[section]
\newtheorem{lem}[thm]{Lemma}
\newtheorem{prop}[thm]{Proposition}
\newtheorem{cor}[thm]{Corollary}
\newtheorem{con}[thm]{Conjecture}
\newtheorem{claim}[thm]{Claim}
\newtheorem{obs}[thm]{Observation}
\newtheorem{definition}[thm]{Definition}
\newtheorem{example}[thm]{Example}
\newtheorem{rmk}[thm]{Remark}
\newcommand{\di}{\displaystyle}
\def\dfc{\mathrm{def}}
\def\cF{{\cal F}}
\def\cH{{\cal H}}
\def\cT{{\cal T}}
\def\C{{\mathcal C}}
\def\cA{{\cal A}}
\def\cB{{\mathcal B}}
\def\P{{\mathcal P}}
\def\Q{{\mathcal Q}}
\def\cP{{\mathcal P}}
\def\cp{\alpha'}
\def\Frk{F_k^{2r+1}}

\newcommand{\F}{\mathcal{F}}

\title{Covering $2$-connected $3$-regular graphs with disjoint paths}
\author{Gexin Yu}
\address{Department of Mathematics\\
The College of William and Mary\\
Williamsburg, VA, 23185, gyu@wm.edu.}

\thanks{The research is supported in part by the NSA grant  H98230-16-1-0316.}

\date{\today}

\maketitle

\begin{abstract}
A path cover of a graph is a set of disjoint paths so that every vertex in the graph is contained in one of the paths.    The path cover number $p(G)$ of graph $G$ is the cardinality of a path cover with the minimum number of paths.    Reed in 1996 conjectured that a $2$-connected $3$-regular graph has path cover number at most $\lceil n/10\rceil$.   In this paper, we confirm this conjecture.


\noindent{\bf AMS subject class:} 05C38, 05C70
\end{abstract}

\section {Introduction}

A {\em path cover} of a graph is a set of disjoint paths that contain all the vertices of the graph. The {\em path cover number} of graph $G$, written as $p(G)$, is the cardinality of a path cover with the minimum number of paths. 

Ore~\cite{O61} initiated the study of path covers.  A graph has path cover number $1$ precisely when it has a Hamiltonian path.  It is well-known that if the minimum degree of an $n$-vertex graph is at least $n/2$ then the graph is Hamiltonian.  Because of its natural connection with hamiltonian graphs, people were interested in the sufficient conditions for a graph to have path cover number at most $k\ge 2$, see, for example, \cite{H88, N75}.   In more recent years, path covers have been used to study other graph parameters,  such as domination numbers~\cite{R96, KS05, LR08},  $L(2,1)$-labelling~\cite{GMW94},  independence number~\cite{H88}, and graphic-TSP~\cite{FRS14}, just to name a few.

Every $n$-vertex graph have a path cover of order at most $n$, and one would imagine that a graph with more edges will require fewer paths to cover. However, an $n$-vertex graph with minimum degree $t$ could have path cover number as high as $n-2t$, for example $K_{t, n-t}$.  Thus,  we are more interested in path cover of regular graphs.   Jackson~\cite{J80} showed that $2$-connected $k$-regular graphs with at most $3k+1$ vertices have a hamiltonian path (actually they have a hamiltonian cycle except the Petersen graph), thus the path cover number is $1$. Magnant and Martin~\cite{MM09} studied path cover numbers of $k$-regular graphs for $k\ge 3$, and they showed that for $k\le 5$, a $k$-regular graph has path cover number at most $n/(k+1)$, which they conjectured to be true for $k>5$.  Note that if every component of a graph $G$ is a clique of $k+1$ vertices, then $p(G)=n/(k+1)$, thus the bound is sharp for general graphs.  As they pointed out,  it is more difficult to find the path cover numbers of connected regular graphs.

The following example gives a general lower bound for the path cover numbers of connected $k$-regular graphs. Take $K_{2, k-1}$ and replace every vertex of degree $2$ with $K_{k+1}^-$ (a $k+1$-clique minus an edge), and call this graph $H$, in which two vertices have degree $k-1$ and the rest have degree $k$. Now let $G$ be the $k$-regular graph with $n$ vertices formed from $\frac{n}{k^2+1}$ pairwise disjoint $H$ by adding $\frac{n}{k^2+1}$ edges to link them in a ring.   It is not hard to see that the path cover number of $G$ is at least $\frac{n(k-3)}{k^2+1}$ for $k\ge 5$. Therefore for $k\ge 13$,  one cannot find a path cover with fewer than $n/(k+4)$ paths in connected $k$-regular graphs (note that the examples are actually $2$-connected).  Some more examples from \cite{OW10, OW11} also show that $n/(k+4)$ paths are necessary.   

Intuitively, one may need more paths to cover the vertices when there are fewer edges in the graphs.  This initiated the study of path covers for connected $3$-regular graphs. Reed~\cite{R96} showed that {\em a connected $3$-regular graph with $n$ vertices has path cover number at most $\lceil n/9\rceil$}, and also gave examples that need $\lceil n/9\rceil$ paths.    He conjectured \cite{R96} that  it suffices to use at most $\lceil n/10\rceil$ paths to cover $2$-connected $3$-regular graphs.  In this article, we confirm this conjecture.

\begin{thm}\label{main}
Every $2$-connected $3$-regular graph with $n\ge 10$ vertices has path cover number at most $n/10$.
\end{thm}

It follows that every $2$-connected $3$-regular graph with at most $20$ vertices contains a hamiltonian path.  
Reed~\cite{R96} gave the following example to show that one cannot improve $\lceil n/10\rceil$ in general:    let $C=u_1v_1u_2v_2\ldots u_kv_k$ be a cycle of $2k$ vertices, let $H$ be the the graph obtained from the Petersen graph by removing an edge, say $uv$, and let $G$ be the graph obtained by replacing edge $u_iv_i$ for $1\le i\le k$ with $H$ so that $u=u_i$ and $v=v_i$.  He claimed that the path cover number of $G$ is $n/10$, based on the observation that one needs a path to cover each $H$.   However, we can use one path to cover two consecutive copies of $H$, thus only need $n/20$ paths to cover $V(G)$.    Here we give infinitely many $2$-connected $3$-regular $n$-vertex graphs whose path cover numbers are at least $n/14$.  

\begin{thm}
There are infinitely many $2$-connected $3$-regular $n$-vertex graphs whose path cover numbers are at least $n/14$.
\end{thm}

\begin{proof}
 Let $G$ be an arbitrary 2-connected 3-regular graph, and let $H$ be the graph obtained from $G$ by replacing each edge of $G$ with a $K_4^-$ (that is, delete the edge, and connect two endpoints of the edge to the two degree-2 vertices on $K_4^-$, respectively).  Then $n(H)=n(G)+4\cdot \frac{3n(G)}{2}=7n(G)$.  We now show that $p(H)\ge n(G)/2=n(H)/14$.

Let $\P$ be a path cover of $H$.  Let $e=uv$ be the edge between $u\in V(G)$ and $v$ in some $K_4^-$.  Then either $uv$ is on a path of $\P$, or $v$ is on some path in $\P$ that contains all vertices of the $K_4^-$.  In the latter case, we may reroute the path so that $v$ is an endpoint, thus extend the path to include the edge $vu$.  Therefore, we may obtain a path decomposition $\P'$ of $G$ (a set of edge-disjoint paths $\P'$ containing all the edges of $G$) with $|\P'|=|\P|$.   Each path in $\P'$ contains a vertex in $G$ as either an internal point or an endpoint, and only when it is an endpoint, the parity of its degree changes when we remove the edges on the path.  But each path can only change the degree parities of at most two vertices in $G$.  As $G$ has $n(G)$ vertices whose degree parity need to be changed, there are at least $n(G)/2$ paths in $\P'$.   Thus,  $\P$ contains at least $n(H)/14$ paths.
\end{proof}

 It is an interesting question to determine the sharp bounds for the path cover numbers of $2$-connected $3$-regular graphs in terms of the orders of the graphs. \\

We will often use the following notation for a path and its segments.
A {\em $k$-path} is a path of $k$ vertices. For a $k$-path $P$, if $G[V(P)]$ contains a spanning cycle, we call it a {\em cyclic $k$-path} or a {\em $k$-cycle}, otherwise {\em non-cyclic}.   A vertex on a non-cyclic path $P$ is called {\em weighty} if it is adjacent to an endpoint of $P$ by an edge not on $P$.   If a path $P$ contains vertex $x$, then we sometimes write $P$ as $P_v$, and let $v^-, v^+$ be the vertices (neighbors) next to $v$ on $P$, respectively.  If the endpoints of $P_v$ are $x$ and $y$, then we also write $P_v$ as $xPy$,  or even as $xPv^-vv^+Py$.  We will use $uPv$ to denote the segment on $P$ from $u$ to $v$. If $v$ is an endpoint of $P_v$, we sometime use $P_vv$ to denote the path $P_v$ with endpoint $v$.  For other notation, we refer to West~\cite{W01}.

\section{Sketch of the proof of Theorem~\ref{main}}

The idea of the proof of the theorem is quite simple.  We consider a specially chosen minimal path cover $\P$, and assign a weight of $10$ to each path in the cover initially. We then redistribute the weights among the paths and show that the final weight on each path is at most its order.  It follows that the total weight is $10|\P|$ on one hand, and at most $n$ on the other hand, therefore $|\P|\le n/10$.  The difficulty lies on the choice of minimal path cover and on the way to redistribute the weights.  Below we give some insights on how we make the choices.

In the minimal path covers we can show that none of the paths are single vertices or contain a spanning cycle.  We may think that the weights of the paths are all on the endpoints, 5 for each.  Let $x$ be an endpoint of a path $P$, and $xu$ be an edge not on $E(P)$.  Clearly, $u$ is not an endpoint of another path, or we will combine the paths into one to get a cover with fewer paths.    We transfer a weight of $2$ from $x$ to $u$.  The vertex $u$ is called {\em weighty} if $u$ is on $P$ and {\em heavy} otherwise.  Then on each path, the four edges incident with the endpoints will send out a weight of $8$ and only a weight of $1$ remains on each endpoint.

Note that heavy vertices are not next to each other on the paths, or we can rearrange the paths to get a cover with fewer paths.  Therefore, if there are no consecutive vertices on paths that are either heavy or weighty, and the number of vertices on a path is not odd, then the final weight on each path is at most the number of vertices on the path, as desired.  Therefore the problematic cases are the existence of consecutive weighty and heavy vertices,  or the number of vertices is odd and every other vertex on a path is weighty or heavy.  They force us to identify more vertices to transfer weights from one path to another, and suggest such vertices to be the ones incident with a vertex whose neighbors are weighty and heavy on the path.  It turns out that we only need to use such vertices, namely {\em P(seudo)E(ndpoint)-vertices}, to transfer a weight of $1$.  PE-vertices share a lot of common features with the endpoints.  For example, a light vertex cannot be next to a heavy vertex on a path, where a light vertex is the neighbor of a PE-vertex not on the path of the PE-vertex.  Light vertices make the proof more complicated.

There are still bad situations that a path may have too much weight. For example, a 3-path with a heavy middle vertex, or a $5$-path $P=xu_1u_2u_3y$ such that $xu_2, xu_3\in E(G)$, or a $6$-path $P=xu_1u_2u_3u_4y$ such that $xu_3, x'u_1, x'u_4\in E(G)$ where $x'$ is an endpoint of another path, or a $7$-path $P=xu_1u_2u_3u_4u_5y$ such that $xu_4, x'u_1, x'u_5\in E(G)$ where $x'$ is an endpoint of another path. Fortunately, we can carefully define the optimal path covers to avoid all those situations.

Each pair of consecutive heavy/weighty vertices on a path contains a neighbor of the endpoints, so there are at most four such pairs on each path.  To show that each path has no more weight than its number of vertices, we show that in each of the bad cases, the path has enough neutral vertices (vertices do not receive weights) and/or PE-vertices.

We define optimal path covers and study their properties in Section~\ref{optimal}. The special vertices (heavy, light, PE-vertices) and their properties are studied in Section~\ref{structure}.  Then in Section~\ref{counting}, we prove the main lemma that the total weight on each path does not exceed its order and finish the proof of the theorem.   

\section{Optimal path covers and their basic properties}\label{optimal}

Let $G$ be a minimum counterexample to Theorem~\ref{main}.  Among all path covers of $G$, choose $\P$ to be an optimal path cover subject to the following:
\begin{enumerate}[(i)]
\item the number of paths is minimized.
\item subject to (i), the number of $1$-paths is minimized.
\item subject to (i)-(ii), the number of $3$-paths and cyclic paths is minimized.  
\item subject to (i)-(iii), the number of bad endpoints is minimized, where an endpoint $x'\in P'\in \P$ is bad if
(v1) $x'$ is adjacent to $u_1, u_4\in P$ and $xu_3\in E(G)$, or
(v2) $x'$ is adjacent to $u_1, u_5\in P$ and $xu_4\in E(G)$, where $P=xu_1u_2u_3u_4\ldots u_ky\in \P-\{P'\}$.
\item subject to (i)-(iv), the number of annoying endpoints is minimized, where an endpoint $x'\in P'=x'u'_1\ldots u'_ly'$ is annoying if $x'u'_{s+1}, x'u_i, xu'_{s-1}, u'_su_{i+1}\in E(G)$ with $P=xu_1u_2\ldots u_ky\in \P-\{P'\}$ and $2\le s\le l-1$.
\item subject to (i)-(v), the number of weighty vertices is minimized.
\item subject to (i)-(vi), for each non-cyclic path $P$, the number of vertices on $P$ between the endpoints and their corresponding furthest neighbors on $P$ is maximized.
\end{enumerate}

We shall call a path cover satisfying the first $t$ conditions above as $\P_t$. Thus $\P$ is $\P_7$, and $\P_0$ is just a path cover to $G$.  Clearly, $\P_{i+1}\subseteq \P_i$, so $\P_{i+1}$ has all the properties that $\P_i$ has.

A {\em net} is a triangle whose three neighbors not on the triangle are distinct.  The following was observed in ~\cite{R96}.

\begin{lem}\label{net}
The graph $G$ contains no net.
\end{lem}

\begin{proof}
For otherwise, let $u_1u_2u_3$ be a triangle with $u_iu_i'\in E(G)$ such that $u_i'$'s are distinct. Then we contract the triangle to a single vertex $u$ and get a graph $G'$.  Now $G'$ has a path cover with at most $|V(G')|/10$ paths, but then we can get a path cover of $G$ by replacing $u$ with a path containing $u_1, u_2, u_3$.
\end{proof}

\begin{lem}\label{basic-facts}The following are true about $\P_1$:
\begin{enumerate}[(1)]
\item Endpoints of different paths in $\P_1$ are not adjacent. In particular, there is no edge between cyclic paths or between a cyclic path and an endpoint of a non-cyclic path.
\item every cyclic path has at least two neighbors not on the path. 
\end{enumerate}
\end{lem}

\begin{proof}
(1) is true because our cover used the minimum number of paths. (2) is true because $G$ has no cut-vertices.  
\end{proof}

The following lemma from~\cite{MM09} says that a path cover subject to (i) and (ii) contains no $1$-paths.  We give an alternative proof here, whose idea will be used to prove more results about path covers.   

\begin{lem}[\cite{MM09}]\label{mmlem}
The path cover $\P_2$ contains no $1$-paths.
\end{lem}

\begin{proof}
Suppose that $P\in \P_2$ consists of vertex $v$.  By Lemma~\ref{basic-facts}, $v$ is not adjacent to an endpoint of another path. We also note that $v$ is not adjacent to an interior vertex on a path $P'$ of order at least $4$, for otherwise, one can easily decompose $P\cup P'$ into two paths, each of order at least $2$. Therefore $v$ must be adjacent to the midpoints of $3$-paths.  Furthermore, if $v$ is adjacent to the vertex $w\in P'=xwy$, then we may rearrange paths to form the paths $xwv$ and $y$ or $vwy$ and $x$. This implies $x$ and $y$ must also be adjacent only to the midpoints of 3-paths.

Let $T$ be the set of $1$-paths and $3$-paths that are  involved in the above rearrangement process.  We consider an auxiliary digraph $D$ whose vertices are the paths in $T$, and there is a directed edge from $P_1\in T$ to $P_2\in T$ if and only if an endpoint of $P_1$ is adjacent to the midpoint of $P_2$.  Clearly, each vertex in $D$ has in-degree at most $1$ and out-degree at least $3$, which is impossible.  Therefore, $\P$ contains no $1$-paths.
\end{proof}

\begin{lem}\label{no-triangle}
The path cover $\P_3$ contains no $1$-paths, $3$-paths, or cyclic paths. 
\end{lem}

\begin{proof}
We call a path {\em bad} if it is cyclic or has order $1$ or $3$.   By Lemma~\ref{mmlem}, we may assume that each bad path in $\P_3$ is a non-cyclic 3-path (note that a cyclic 3-path must be a net) or a cyclic path with order at least $4$.

Let $P$ be a bad path in $\P_3$ and $x\in P$ be a potential endpoint of $P$, which is an endpoint if $P$ is non-cyclic, or any vertex on $P$ if $P$ is cyclic.  Suppose that $xw\in E(G)$ with $w\in Q\in \P_3-\{P\}$.  Then $Q$ is not cyclic, and $Q-w$ splits into two paths $Q_1$ and $Q_2$.    In fact, the path $P'$ obtained by concatenating $P, w, Q_i$ cannot be cyclic (as then $P+Q$ would have a Hamilton path contradicting minimality) or have length less than $4$. Thus both $Q_1$ and $Q_2$ must be bad by the minimality of the cover. Furthermore, neither $Q_1$ nor $Q_2$ is a cyclic 3-path, or we would have a net or a path cover with fewer paths.    We shall call $w$ a {\em special vertex} on $Q$, and $Q_1, Q_2$ {\em bad components} on $Q$.

Now, for $i\in \{1,2\}$, provided $Q_i$ has order more than one,  replacing $P,Q$ in $\P$ with $Q_i$ and $P'=P+w+Q_{3-i}$ gives a new minimal path cover. We can repeat our argument using $Q_i$ in the place of $P$ and any other non-bad path of the new cover other than $P'$ in the place of $Q$.

We build a directed graph whose vertices are the paths in $\P$, and a family $\F$ of subpaths of these paths as follows.
\begin{itemize}
\item[(A)] The set $\F_0$ consists of bad paths in $\P_3$; and we add a directed edge from $P\in \F_0$ to $Q\in \F_0$ if a potential endpoint of $P$ is adjacent to a special vertex on $Q$ (note that this can only happen if $Q$ is a non-cyclic $3$-path);

\item[(B)] If an endpoint of a Hamilton path on the vertex set of a bad path $P\in \F_0$ is adjacent to $w$ on a non-bad path $Q\in \P$, we add to $\F_1$ all the bad components of $Q-w$ which do not have order $1$, and add to our digraph an edge from $P$ to $Q$;

\item[(C)] For $i\ge 1$,  if an endpoint of a Hamilton path on the vertex set of some bad path $P\in \F_i$ is adjacent to a special vertex $w$ of some non-bad path $R\in \P_3$, we add to $\F_{i+1}$ all the components of $R-w$ which do not have order $1$,  and we add to the digraph the edge from the path in $\P_3$ that contains $P$ as a bad component described in (B) to $R$.  Note that multi-edges are allowed, but we only allow one directed edge implied by the middle vertex of each 5-path.
\end{itemize}

We let $\F$ be the union of the $\F_i$.  By definition, the in-degree of a path equals to the number of special vertices on it.  Note that a cyclic bad path or component does not contain special vertices.  It follows that if a non-bad path $P$ contains two special vertices $w_1$ and $w_2$, then the bad component in $P-w_1$ that contains $w_2$ must be a $3$-path, and the bad component in $P-w_2$ that contains $w_1$ must also be a $3$-path, so $P$ must be a non-cyclic $5$-path.  Therefore, the in-degrees of $5$-paths are at most $2$ and all other paths are at most $1$. Note that there may be isolated vertices in the digraph.

Now we count the out-degrees. The out-degree of a path $P$ equals to the number of edges that connect one endpoint of a bad component of order more than $1$ and a special vertex not on $P$.  Let $Q_1, Q_2$ be the two bad components of a path $P\in \P_3$ in the digraph.

If $Q_1$ and $Q_2$ both have order $1$, then $P$ is a bad $3$-path.  By (A), $P$ has out-degree $4$.  So let $Q_1$ have order more than $1$.  Note that $Q_1$ has at least two edges out of $Q_1$ (as $G$ is $2$-connected), one of which is not adjacent to the special vertex on $P$.

If $Q_1$ and $Q_2$ are both cyclic or have order $1$, then there can be no edge between them, as if $Q'=Q_1\cup Q_2$ has a Hamiltonian cycle, we can rearrange $P\cup Q$ into one path, contradicting the minimality of the cover, and otherwise $P'=P+w$ and $Q'$ are both non-cyclic and  we contradict the minimality of the number of bad paths in the cover.    So $P$ has out-degree at least $1$ (actually $2$ if both $Q_1, Q_2$ are cyclic).

Now, if $Q_1$ has order three and $Q_2$ is cyclic or has order $1$, then (a) the endpoint of $Q_1$ which is an endpoint of $Q$ cannot be adjacent to any vertex on $Q_2$ or we could find a Hamilton path on $P\cup Q$ contradicting the minimality of the cover, and (b) the other endpoint $x'$ of $Q_1$ can be adjacent to none of the vertices on $Q_2$ or we could find a Hamiltonian path $P'$ on $P\cup Q_2+w+x'$, which together with $Q′=Q_1-x'$ contradicts the minimality of the cover.   Similar arguments show that if $Q_1$ and $Q_2$ both have order three then there are no edges joining their endpoints.  So $P$ has out-degree at least $3$.

Since the out-degrees of the paths are as large as their in-degrees, and the bad paths have higher out-degrees than their in-degrees, such a digraph does not exist, a contradiction.
\end{proof}

From now on, we assume that $\P_3$ consists of non-cyclic paths with order other than $1$ and $3$.

\begin{lem}\label{no-7}
There are no bad endpoints described in (iv) in $\P_4$.
\end{lem}

\begin{proof}
Suppose otherwise.  Consider an endpoint $x'\in P'=x'u_1'u_2'\ldots u_t'y'$ in (v1).  We replace $P,P'$ with $P'x'u_1x$ and $u_2Py$.   We lose $x'$, and do not create cyclic paths or we would contradict the minimality of the cover.  We do not gain a new bad endpoint described in (v1) because (a) $x$ is not adjacent to $u_1'$ or we could rearrange $P'+P$ into one path $P'u_1'xu_3u_2u_1x'u_4Py$, and (b) $u_2$ is not adjacent to $u_5$ or we could rearrange $P'+P$ into one path $P'x'u_4u_3xu_1u_2u_5Py$.  We also do not gain a new bad endpoint described in (v2) because (a) $x$ cannot be, or $u_2u_6, xu_7\in E(G)$, which allows us to reroute $P,P'$ into one path $P'x'u_1u_2u_6u_5u_4u_3xu_7Py$, and (b) $u_2$ cannot be, or $u_2u_3', xu_2'\in E(G)$, which allows us to reroute $P, P'$ into one path  $P'u_3'u_2u_1x'u_1'u_2'xu_3u_4Py$.

Consider an endpoint $x'\in P'$ in (v2).  If $u_5Py$ is a $3$-path, then we replace $P,P'$ with $u_6y$ and $xPu_5x'P'$.  We lose $x'$, and do not create cyclic paths, but clearly do not gain a new bad endpoint, as $u_1$ is now adjacent to a vertex on the path, and $u_6$ has at most one neighbor on other paths.  If $u_5Py$ is not a $3$-path, then we replace $P,P'$ with $u_5Py$ and $u_2u_3u_4xu_1x'P'$.  We lose $x'$, and do not create cyclic paths.  We do not gain new bad endpoints, since $u_2, u_5$ cannot be as they have at most one neighbor on other paths, and no other vertex can be adjacent to $x'$ as it is already adjacent to $u_5, u_1, u_1'$, and no vertex from other path can be a bad endpoint (to $u_5Py$) as $u_5$ has only one neighbor on the path. 
\end{proof}

\begin{lem}\label{annoying}
There are no annoying endpoints described in (v) in $\P_5$.
\end{lem}

\begin{proof}
Let $x'\in P'=x'u'_1\ldots u_l'y'\in \P_5$ be an annoying endpoint. Then for $P=xu_1\ldots u_ky\in \P_5-\{P'\}$, $x'u_i, x'u'_{s+1}, xu'_{s-1}, u'_su_{i+1}\in E(G).$  Note that $P,P'$ can be decomposed into $yPu_{i+1}u_s'P'y'$ and cyclic path $x'P'u_{s-1}'xPu_i$. So $y, y'$, and endpoint of the paths in $\P_5-\{P,P'\}$ cannot have neighbors on the cyclic path. 

{\bf Case 1.} $s=3$ and $u_1'u_{i+2}\in E(G)$.  We replace $P, P'$ with $P_1=x'u_1'$ and $P_1'=yPxu_2'P'y'$.  Since $P_1$ is a $2$-path and $xy, xy'\not\in E(G)$,  $P_1, P_1'$ are not $1$-paths, $3$-paths, or cyclic paths, and none of the endpoints ($x', u_1', y, y'$) becomes bad or annoying.  But we have fewer annoying endpoints, a contradiction.  

{\bf Case 2.} $s>3$, or $s=3$ but $u_1'u_{i+2}\not\in E(G)$.  We replace $P,P'$ with $P_2=u_iPxu_{s-1}'P'x'u_{s+1}'P'y'$ and $P_2'=u_s'u_{i+1}u_{i+2}Py$.   Note that none of $P_2, P_2'$ can be $1$-paths, $3$-paths, or cyclic paths. Since $u_i$ has at most one neighbor on paths other than $P_2$, $u_i$ is not bad.  Since $u_{s-1}'$ is nto next to $u_i$ (the endpoint of $P_2$), $u_s'$ is not bad.  Since $u_s'$ has only one neighbor on $P_2'$, it is not annoying.  Note that $u_i$ is annoying only if $u_1'u_{i+2}\in E(G)$ and $s=3$ (so that $u_{s-1}'=u_2'$), so $u_i$ is not annoying.  Since $y, y'$ have no neighbors in $u_iP_2x'$, they cannot become new bad or annoying endpoints.  Therefore, we have fewer annoying endpoints, a contradiction.  
\end{proof}

\begin{lem}\label{no-two-consecutive}
Let $P=xu_1u_2\ldots u_ky\in \P_7$ be a non-cyclic path so that $xu_i,xu_j\in E(G)$ with $1<i<j\le k$.  Then $j\not=i+1$, and the neighbors of $u_{i-1}$ and $u_{j-1}$ are on $P$.  Furthermore, if $y$ has no neighbors on $xPu_j$, then the neighbors of $u_{j-1}$ and $u_{i-1}$ must be on $xPu_j$.
\end{lem}

\begin{proof}
If $j=i+1$ and $i\not=2$, then $xu_iu_{i+1}$ is a net, a contradiction to Lemma~\ref{net}. If $j=i+1=3$, then $u_1$ will be a better endpoint than $x$ subject to (vi) (with fewer weighty vertices) or (vii) (weighty neighbors are further away from the endpoints).   Let $j>i+1$.  If $u_{i-1}$ (or $u_{j-1}$) has a neighbor outside of $P$, then we can reroute $P$ so that $u_{i-1}$ (or $u_{j-1}$) is an endpoint, which would give fewer weighty vertices,  a contradiction to the optimality of $\P$.   If $y$ has no neighbors on $xPu_j$, and $u_{j-1}$ (or $u_{i-1}$)  has a neighbor $u_t$ with $t>j$, but then $u_{j-1}$ (or $u_{i-1}$) is a better endpoint than $x$ subject to (vii) (and we do not change the number of vertices between $y$ and its furthest neighbor).
\end{proof}

\section{Properties of heavy, light, and PE-vertices}\label{structure}

In this section, we study the properties of some special vertices on the paths in $\P$.  

\begin{definition}
Let $u$ be an endpoint of a path $P\in \P_4$ and $uv\in E(G)-E(P)$.  Then $v$ is called a heavy vertex if $v\not\in V(P)$ (and a weighty vertex is $v\in V(P)$). 
\end{definition}

\begin{definition}
Let $uv$ be an edge between $u=u_i\in P=xu_1\ldots u_ky\in \P_4$ and $v\in P_v\in \P_4-\{P\}$. Then $u$ is called a PE-vertex (aka, pseudo-endpoint) and $v$ is called a light vertex if one of the following is true
\begin{enumerate}
\item[(1a)] $xu_{i+1},yu_{i-1}\in E(G)$; or
\item[(1b)] $xu_{i+1}\in E(G)$, and $u_{i-1}$ is heavy;  or $yu_{i-1}\in E(G)$ and $u_{i+1}$ is heavy; or
\item[(1c)] both $u_{i-1}$ and $u_{i+1}$ are heavy.
\end{enumerate}
A vertex is {\em neutral} if it is not heavy or light or weighty.
\end{definition}

 Note that a PE-vertex is also a neutral vertex. Also note that if $u$ is a PE-vertex defined in (1a) and (1b), then $P_u$ can be rerouted so that $u$ (and $x$ or $y$) is an endpoint of the path.

\begin{lem}\label{neutral-end}
Let $u\in P\in \P, v\in P_v\in \P-\{P\}$ with $uv\in E(G)$. If $P=xPy$ can be rerouted so that $u$ is an endpoint, then $v$ cannot be an endpoint or a PE-vertex, unless $xv^-, yv^+\in E(G)$. Consequently, if $u$ and $y$ are the endpoints, then $ux\in E(G)$ or $u$ is neutral. 
\end{lem}

\begin{proof}
If $P_v$ can also be routed so that $v$ is an endpoint, then $P, P_v$ can be combined into one path, a contradiction. So $v$ cannot be an endpoint or a PE-vertex defined as in (1a) or (1b).  

Let $v$ be a PE-vertex defined as in (1c). Let $P_v=x_vP_vv^-vv^+P_vy_v$.  Then $v^-$ and $v^+$ are heavy. We assume that $v^-x_s, v^+x_t\in E(G)$, where $x_s, x_t$ are endpoints of $P_s, P_t\in \P-\{P_v\}$, respectively.  If one of $P_s$ and $P_t$, say $P_s$, is not $P$, then we can decompose $P, P_s, P_v$ into two paths: $P_sx_sv^-P_vx_v$ and $y_vP_vvuP$, a contradiction.  So $P_s=P_t=P$.  If $y$ (and by symmetry, $x$) has only one neighbor on $P$, then $y$ must be  the other endpoint when $u$ is an endpoint of $P$, thus $yv^+\not\in E(G)$ (and similarly, $yv^-\not\in E(G)$),  or $P, P_v$ can be combined into one path $y_vP_vv^+yPuvP_vx_v$.  It follows that $xv^-, xv^+\in E(G)$, and thus $x$ has only one neighbor on $P$, a contradiction. So both $x$ and $y$ have at least two neighbors on $P$. Then we must have $xv^-, yv^+\in E(G)$.   

When $u$ and $y$ are the endpoints, $v$ cannot be an endpoint or a PE-vertex, so $u$ is not heavy or light and $uy\not\in E(G)$. Then $u$ is neutral or weighty, and when it is weighty, we have $xu\in E(G)$. 
\end{proof}

\begin{cor}\label{neutral}
Let $P=xu_1\ldots u_ky\in \P$ and $1\le i<j\le k$. Then
\begin{itemize}
\item if $xu_j\in E(G)$, then $u_{j-1}$ is neutral;
\item if $xu_{i+1}\in E(G)$ and $u_iu_j\in E(G)$, then $u_{j-1}$ is neutral or $xu_{j-1}\in E(G)$;
\item if $xu_{j+1}\in E(G)$  and $u_iu_j\in E(G)$, then  $u_{i+1}$ is neutral or $xu_{i+1}\in E(G)$;
\item if $xu_{j-1}\in E(G)$  and $u_iu_j\in E(G)$, then $u_{i-1}$ and $u_{i+1}$ are neutral or adjacent to $x$;
\item if $xu_{i+1}, xu_{j+1}, u_iu_j\in E(G)$, then $u_1, u_{j-1}$ are neutral.
\end{itemize}
\end{cor}

\begin{cor}\label{no-neighbor-of-y}
Let $P=xu_1\ldots u_ky\in \P$. If $xPu_i$ is cyclic and $u_{i+1}$ is heavy or light, then a vertex $u\in xPu_{i-1}$ is adjacent to $y$ or $u_j$ with $yu_{j-1}\in E(G)$ only when $u_{i+1}$ is light and is adjacent to $v_s\in P'=x'P'v_{s-1}v_sv_{s+1}P'y'$ such that $xv_{s-1},yv_{s+1}\in E(G)$.
\end{cor}

\begin{proof}
Under the condition, $P$ can be rerouted so that $u_{i+1}$ is an endpoint.  So the statement follows from Lemma~\ref{neutral-end}. 
\end{proof}

\begin{lem}\label{25-well}
PE-vertices form an independent set. Consequently, no light vertex is a PE-vertex. 
\end{lem}

\begin{proof}
Let $u, u'\in E(G)$ be PE-vertices such that $uu'\in E(G)$ with $u\in P_u=xPu^-uu^+Py$ and $u'\in P_{u'}=x'Pu'^-u'u'^+Py'\in \P-\{P\}$.   By Lemma~\ref{neutral-end}, we may assume that $u, u'$ are PE-vertices defined as in (1c), thus assume that $u^-, u^+, u'^-, u'^+$ are adjacent to endpoints $s\in P_s, t\in P_t, s'\in P_{s'}, t'\in P_{t'}$, respectively, where $P_s, P_t\not=P_u$ and $P_{s'}, P_{t'}\not=P_{u'}$.

First assume that none of the pairs $(P_s, P_{s'}), (P_s, P_{t'}), (P_t, P_{s'}), (P_t, P_{t'})$ contains two different paths. Then  $P_s=P_{s'}=P_{t'}=P_t$.  Without loss of generality, we may assume that $s, s'$ are the endpoints of $P_s$. Then $P_u, P_{u'}$ and $P_s$ can be combined into two paths: $xP_uu^-sP_ss'u'^-P_{u'}x'$ and $yP_uuu'P_{u'}y'$, a contradiction.  Therefore, without loss of generality, we assume that $P_s\not=P_{s'}$.

If $P_s\not=P_{u'}$ and $P_{s'}\not=P_{u}$, then we reach a contradiction by combining $P_u, P_s, P_{u'}, P_{s'}$ into three paths: $P_su^-P_ux, P_{s'}u'^-P_{u'}x'$ and $yP_uuu'P_{u'}y'$.  Thus,  we may assume that $P_s=P_{u'}$ and let $s=x'$.

If $P_{s'}\not=P_u$,  we can decompose $P_{s'}, P_u, P_{u'}$ into fewer paths: $xP_uu^-x'P_{u'}u'^-s'P_{s'}$ and $yPuu'P_{u'}y'$, again a contradiction.  Therefore, we may assume that $P_{s'}=P_u$.   By symmetry, we also know that $P_t=P_{u'}$ and $P_{t'}=P_u$.

Let $xu'^-\in E(G)$.   If $x'u^+\in E(G)$ (or by symmetry $y'u^-\in E(G)$), then we reach a contradiction by combing $P_u$ and $P_{u'}$ into one path $yP_uu^+x'P_{u'}u'^-xP_uuu'P_{u'}y'$. Thus,  we let $x'u^-, y'u^+\in E(G)$. But we again can combine the two paths into one path $yP_uu^+y'P_{u'}u'uP_uxu'^-P_{u'}x'$.
\end{proof}

\begin{lem}\label{paired-cycle}
Let $P=xu_1u_2\ldots u_ky\in \P_4$. Assume that for some $1<s<i<t<k$,  the subgraphs induced by $V(xPu_i)$ and $V(u_{i+1}Py)$ contain spanning paths so that $u_s$ and $u_t$ are the endpoints, respectively.  If $u_s, u_t$ are heavy or light, then
\begin{enumerate}[(a)]
\item $u_s$ and $u_t$ are both light; or
\item $u_s, u_t$ are heavy and adjacent to a same endpoint of $P'\in \P_4-\{P\}$; or
\item $u_s$ is heavy and $u_t$ is light (or by symmetry $u_t$ is heavy and $u_s$ is light) with $x_wu_s, vu_t\in E(G)$, where $x_w$ is an endpoint of $P_w\in \P_4-\{P\}$ and $v\in P_v=x_vP_vv^-vv^+P_vy_v$, such that
\begin{itemize}
\item[(c1)] $P_w=P_v$, and $x_vu_s, x_vv^+\in E(G)$ and $v^-$ is adjacent to $x$ or $y$, or
\item[(c2)] $P_w\not=P_v$, and $v^-, v^+$ are adjacent to $x, y$ or $x_w$.
\end{itemize}
\end{enumerate}
Consequently,  let $xPu_i$ be cyclic, then
\begin{enumerate}
\item if $u_{i+1}$ is heavy, then $xPu_i$ contains at most one heavy or light vertex; and
\item if $u_{i+1}$ is light, then $xPu_i$ contains at most one heavy vertex.
\end{enumerate}
\end{lem}

\begin{proof}
Let $P_1, P_2$ be the spanning paths on $V(xPu_i)$ and $V(u_{i+1}Py)$ so that $u_s, x'$ and $u_t, y'$ are endpoints, respectively.  We may assume that at least one of $u_s, u_t$ (say $u_s$) is heavy, or we have (a).  Let $x_wu_s\in E(G)$ from the endpoint $x_w\in P_w\in \P_4-\{P\}$ and $vu_t\in E(G)$ from $v\in P_v=x_vP_vv^-vv^+P_vy_v\in \P_4-\{P\}$.

Assume first that $P_v$ can be rerouted such that $v$ is an endpoints.   If $v$ is heavy, then we must have (b),  or $P, P_w, P_v$ can be replaced with paths $P_wwuP_1$ and $P_vvu_{i+1}P_2$ to obtain a better path cover.  So let $v$ be light, and by symmetry let $x_vv^+\in E(G)$ and $v^-$ be heavy and adjacent to an endpoint $z\in P_z\in \P-\{P_v\}$.   Then $x_v=x_w$, or we replace $P, P_v, P_w$ with $y_vPv^+x_vP_vvu_tP_2y'$ and $y_wP_wx_wu_sP_1x'$ (when $x_w\not=y_v$) or $x'P_1u_sP_1y_vP_vv^+x_vP_vvu_tP_2y'$ (when $x_w=y_v$).  Now $z$ must be $x$ or $y$, as in (c1), or we could get a cover with fewer paths: $P_1'=P_zzv^-P_vx_vu_sP_1x'$ and $P_2'=y_vP_vvu_tP_2y'$.

Now assume that both $v^-$ and $v^+$ are heavy.  We may assume that $v^-$ is adjacent to an endpoint $z\in P_z\in \P-\{P_v\}$. If $z\not\in \{x,y,x_w\}$, then we can replace $P, P_w, P_z, P_v$ with paths $P_wwuP_1, P_2u_{i+1}vP_vy_v$ and $P_zzv^-P_vx_v$, a contradiction. So $v^-$, and by symmetry $v^+$, is adjacent to $x, y$ or $x_w$, as in (c2).  
\end{proof}

\begin{lem}\label{no-consecutive}
The subgraph induced by the set of heavy and light vertices contains no edges.
\end{lem}

\begin{proof}
By Lemma~\ref{25-well}, two vertices that are heavy or light are adjacent only if they are consecutive vertices on a path in $\P$.
Let $u_i, u_{i+1}$ be two vertices on $P$ that are heavy or light.  We may assume that $u_{i+1}$ is adjacent to $x_v$ or the vertex $v$ on $P_v=x_vP_vv^-vv^+P_vy_v\in \P-\{P\}$. As $xPu_i$ and $u_{i+1}Py$ contain spanning trees such that $u_i$ (and $x$) and $u_{i+1}$ (and $y$) are endpoints, respectively, by Lemma~\ref{paired-cycle}, the following are the possible cases:

{\bf Case 1.} both $u_i$ and $u_{i+1}$ are heavy.  Then they are adjacent to the same endpoint $x_v\in P_v$.  In this case, $x_vu_iu_{i+1}$ is a net, which cannot occur by Lemma~\ref{net}.

 {\bf Case 2.} $u_i$ is heavy and $u_{i+1}$ is light.   Then by Lemma~\ref{paired-cycle}, we consider the following cases.
 
 {\bf Case 2.1.} $x_vu_i, xv^-, vu_{i+1}, x_vv^+\in E(G)$, or $x_vu_i, yv^-, vu_{i+1}, x_vv^+\in E(G)$. In the former case, $x_v$ is an annoying endpoint, which by Lemma~\ref{annoying} cannot exist, and in the latter case, we can combine $P, P_v$ into one path: $xPu_ix_vP_vv^-yPu_{i+1}vP_vy_v$.

 {\bf Case 2.2.} $vu_{i+1}, xv^-, x_wu_i\in E(G)$, where $x_w$ is an endpoint of $P_w\in \P-\{P,P_v\}$, and $v^+$ is adjacent to $y$ or $x_w$.  Then $P, P_v, P_w$ can be combined into two paths: $P_wx_wu_iPxv^-P_vx_v$ and $y_vP_vvu_{i+1}Py$, a contradiction.

{\bf Case 3.} both $u_i$ and $u_{i+1}$ are light.  Then $wu_i, vu_{i+1}\in E(G)$ for PE-vertices $w\in P_w=x_ww_1\ldots w_{s-1}ww_{s+1}\ldots w_ly_w\in \P-\{P\}$ and $v\in P_v=x_vv_1\ldots v_{t-1}vv_{t+1}\ldots v_my_v\in \P-\{P\}$.

 {\bf Case 3.1} $w$ and $v$ are PE-vertices defined as in (1a) or (1b).  Then $P_w$ can be rerouted so that $w$ and $y_w$ are endpoints, and $P_v$ can be rerouted so that $v$ and $y_v$ are endpoints. 

\begin{itemize}
\item If $P_w\not=P_v$, then $P, P_v, P_w$ can be combined into paths $P_wwu_iPx$ and $P_vvu_{i+1}Py$, a contradiction.
\item If $P_w=P_v$, then $x_w$ or $y_w$ cannot be adjacent to two weighty vertices, by Lemma~\ref{no-two-consecutive}, so we may assume that $x_ww_{s+1}\in E(G)$ and $v=w_j$ so that $y_ww_{j-1}\in E(G)$. Clearly, $j<s$, or $P, P_w$ can be combined into one path $xPu_iwP_wx_ww_{s+1}P_ww_{j-1}y_wP_ww_ju_{i+1}Py$. By definition, $w_{j+1}$ is heavy is adjacent to an endpoint $z\in P_z\not=P_w$, so $P, P_w, P_z$ can be decomposed into paths $P_zzw_{j+1}P_ww_su_iPx$ and $yPu_{i+1}w_jP_wx_ww_{s+1}P_wy_w$, a contradiction.
\end{itemize}

{\bf Case 3.2} $w$ is a PE-vertex defined as in (1a) or (1b), and $v$ is a PE-vertex defined as in (1c).  Let $x_ww_{s+1}\in E(G)$.  By definition, $v_{t-1}$ and $v_{t+1}$ are heavy, then at least one of them, say $v_{t-1}$, is not adjacent to $x_w$.  Let $v_{t-1}$ be adjacent to the endpoint $z\in P_z\not=P_v$. When $P_z=P$, we assume that $z=y$.
\begin{itemize}
\item  If $P_v\not=P_w$,  then $P, P_w, P_v, P_z$ can be decomposed into fewer paths: $$P_zzv_{t-1}P_vx_v,\   y_vP_vvu_{i+1}Py,\  P_wwu_iPx.$$
\item If $P_v=P_w$ and $v=w_j$ with $j>s$, then $P, P_w, P_z$ can also be decomposed into fewer paths: $P_zzw_{j-1}P_ww_{s+1}x_wP_wwu_iPx, \quad  y_wP_ww_ju_{i+1}Py.$
\item If $P_v=P_w$ and $v=w_j$ with $j<s$, then $P, P_w, P_z$ can be decomposed into fewer paths: $P_zzw_{j-1}P_wx_ww_{s+1}P_wy_w, \quad  xPu_iwP_ww_ju_{i+1}Py.$
\end{itemize}
There is a contradiction in each of the cases.

{\bf Case 3.3.} Both $w$ and $v$ are PE-vertices defined as in (1c). 

Let $z_1w_{s-1}, z_2w_{s+1}, z_3v_{t-1}, z_4v_{t+1}\in E(G)$ such that $z_i$ is an endpoint of $P_{z_i}\in\P$, respectively.

Let $P_w\not=P_v$.  As each endpoint is adjacent to at most two heavy vertices, we may assume that $P_{z_1}\not=P_{z_4}$ or $P_{z_1}=P_{z_4}$ and $z_1, z_4$ are the endpoints.  Then $P, P_w, P_v, P_{z_1}, P_{z_4}$ can be decomposed into fewer paths: $P_{z_1}z_1w_{s-1}P_wx_w,\quad P_{z_4}z_4v_{t+1}P_vy_v,\quad   y_wP_wwu_iPx,\quad  x_vP_vvu_{i+1}Py$.

Let $P_w=P_v$.  Assume that $v=w_t$ for some $t>s$.  Note that $P_{z_1}=P_{z_4}$ and $z_1=z_4$, or $P, P_w, z_1, z_4$ can be decomposed into fewer paths:  $xPu_iw_sP_ww_tu_{i+1}Py, P_{z_1}z_1w_{s-1}P_wx_w$ and $P_{z_4}z_4w_{t+1}P_wy_w$.  
We may also assume that $P_{z_2}=P_{z_4}$ (and similarly, $P_{z_3}=P_{z_4}$), or $P, P_w, P_{z_2}, P_{z_4}$ can be combined into fewer paths:  $xPu_iwP_wx_w,   P_{z_2}w_{s+1}P_wvu_{i+1}Py,   P_{z_4}z_4w_{t+1}P_wy_w$.
Therefore $P_{z_1}=P_{z_2}=P_{z_3}=P_{z_4}$, and $z_1$ and $z_2$ are the endpoints.  But then $P, P_w$ and $P_{z_1}$ can be combined into two paths:  $xPu_iw_sP_ww_{t-1}z_2P_{z_2}z_1w_{s-1}P_wx_w$ and $yPu_{i+1}w_tP_wy_w$, a contradiction.
\end{proof}

\begin{lem}\label{chord}
Let $P=xu_1u_2\ldots u_ky\in \P_4$ and $u_iu_j\in E(G)$ for some $i,j$ with $j\not=i-1,i+1$.
\begin{enumerate}[(a)]
\item Let $u_{i-1}$ and $u_{i+1}$ be heavy.  Then $u_{j-1}$ is not weighty,  and $u_{j-1}$ is heavy only if $u_{i-1}$ and $u_{j-1}$ are adjacent to the same endpoint of $P'\in \P-\{P\}$;  Similarly,  $u_{j+1}$ is heavy only $u_{i+1}$ and $u_{j+1}$ are adjacent to the same endpoint of $P''\in \P-\{P\}$.  Furthermore, if both $u_{j-1}$ and $u_{j+1}$ are heavy, then $P'=P''$.

\item  Let $xu_{i+1}\in E(G)$ and $u_{i-1}$ be heavy or $yu_{i-1}\in E(G)$.  Then $u_{j-1}$ is neutral when $j>i$, and $u_{j+1}$ is neutral when $j<i$.  
\end{enumerate}
\end{lem}

\begin{proof}
(a)   As $u_{i-1}$ and $u_{i+1}$ are heavy,  we may assume that they are adjacent to endpoints $x', x''$ of $P', P''\in \P-\{P\}$, respectively.

First note that $u_{j-1}$ is not weighty. If $xu_{j-1}\in E(G)$, then $P,P'$ can be combined into one path $P'x'u_{i-1}Pxu_{j-1}Pu_iu_jPy$, a contradiction.  If $yu_{j-1}\in E(G)$, then $P, P''$ can be combined into one path $P''x''u_{i+1}Pu_{j-1}yPu_ju_iPx$, a contradiction again.

Note that $u_{i-1}$ and $u_{j-1}$ are the endpoints of the spanning paths $xPu_{i-1}$ and $u_{j-1}Pu_iu_jPy$, respectively.  By Lemma~\ref{paired-cycle}, if $u_{j-1}$ is not neutral, then it is adjacent to $x'$, as claimed,  or it is light as in (c1) or (c2) in Lemma~\ref{paired-cycle}.

If it is the case as (c1), then $u_{j-1}$ is adjacent to a PE-vertex $v\in P'$ such that $x'v^+\in E(G)$. Now $P', P'', P$ can be combined into fewer paths: $P''u_{i+1}Pu_{j-1}vP'x'v^+P'y'$ and $xPu_iu_jPy$. So we may assume that it is the case as (c2).  Then $u_{j-1}$ is adjacent to a PE-vertex $v\in P_v=x_vP_vv^-vv^+P_vy_v\in\P-\{P,P'\}$ such that $v^-, v^+$ are adjacent to $x,y$ or $x'$. We may assume that $v^-$ is adjacent to $x$ or $y$. Then $P', P_v, P$ can be combined into two paths in either case: in the former case $x_vP_vv^-Pu_{i-1}x'P'$ and $y_vP_vvu_{j-1}Pu_iu_jPy$, and in the latter case, $P'x'u_{i-1}Px$ and $x_vP_vv^-yPu_ju_iPu_{j-1}vP_vy_v$, a contradiction. 

 For the furthermore part, if both $u_{j-1}$ and $u_{j+1}$ are heavy and adjacent to different paths, say $P'$ and $P''$, then we can replace $P, P', P''$ with $P'x'u_{j-1}Pu_{i+1}x''P''$ and $xPu_iu_jPy$, a contradiction.

(b) When $j>i$, $P$ can be rerouted so that $u_{j-1}, y$ (or $u_{j+1},y$ when $j<i$) are endpoints, it follows from Lemma~\ref{neutral-end} that $u_{j-1}$ is neutral or $xu_{j-1}\in E(G)$; but in the latter case, $P$ can be rerouted so that $u_{i-1}$ and $y$ are endpoints, thus by the same lemma, $u_{i-1}$ cannot be heavy or adjacent to $y$, a contradiction. 
\end{proof}

\begin{lem}\label{no-c}
Let $P=xu_1u_2\ldots u_ky\in \P$. If $xu_3\in E(G)$ and $u_4$ is heavy, then $u_1, u_2$ are neutral.
\end{lem}

\begin{proof}
By Lemma~\ref{no-two-consecutive}, $u_2x\not\in E(G)$,  and then by Lemma~\ref{neutral-end},  $u_2$ is neutral, and $u_1y\not\in E(G)$. Assume that $u_1$ is not neutral.  Then $u_1$ must be light or heavy. If $u_1$ is heavy, then by Lemma~\ref{paired-cycle}, $u_1$ and $u_4$ are adjacent to the same endpoint of a path, thus we have a bad endpoint, a contradiction to Lemma~\ref{no-7}. So $u_1$ must be light.

Let $u_1$ be adjacent to the PE-vertex $v\in P_v=x_vP_vv^-vv^+Py_v\in \P-\{P\}$.  Let $u_4$ be adjacent to an endpoint $x_w\in P_w\in \P-\{P\}$.  The following cases (c1) and (c2) from Lemma~\ref{paired-cycle} must be true.

(c1). $P_w=P_v$, and $x_vu_4, x_vv^+\in E(G)$ and $v^-$ is adjacent to $x$ or $y$. If $v^-x\in E(G)$, then $P,P_v$ can be combined into one path $yPu_4x_vP_vv^-xu_3u_2u_1vP_vy_v$, a contradiction. So $yv^-\in E(G)$.

Note that $P, P_v$ can be replaced by $y_vP_vvu_1u_2u_3x$ (or $y_vP_vvu_1xu_3u_2$) and cyclic path $x_vPv^-yPu_4$, so $x, u_2, y_v$ have no neighbors on $x_vP_vv^-$ and $u_4Py$, or we can combine $P, P_v$ into one path.  

If $x$ has no neighbors on $P_v$ and $u_4Py$ is not a $3$-path, let $P_1=y_vPv^+x_vP_vvu_1u_2u_3x$ and $P_2=u_4Py$; If $x$ has no neighbors on $P_v$ and $u_4Py$ is a $3$-path, let $P_1=y_vPx_vu_4Px$ and $P_2=u_5y$;  If $x$ has a neighbor on $P_v$, let $P_1=xu_3u_4Py$ and $P_2=y_vP_vv^+x_vP_vvu_1u_2$. Note that each of $x, x_v$ is adjacent to at least one weighty vertex in $P, P_v$, respectively. We replace $P, P_v$ with $P_1, P_2$ in the corresponding cases, and claim that there are fewer weighty vertices in the new cover.  In the first two cases, $x$ has one weighty neighbor on $P_1$ and $u_4$ or $u_5$ has none in $P_2$, and in the last case, $x$ has no weighty neighbors on $P_1$, and $u_2$ has at most one weighty neighbor on $P_2$.  Clearly, we do not add bad paths (1-, 3- or cyclic paths) to the cover.  To obtain a contradiction, we show below that we do not create new bad or annoying endpoints. 

Only $x$ in the last case could be a bad endpoint described in (iv), and when it is, we must have $u_2g, xg^-\in E(G)$, where $P_v=x_vP_vg^-gv^-vv^+P_vy_v$;  but in this case, we can combine $P, P_v$ into one path $y_vP_vvu_1xu_3u_2gv^-yPu_4x_vP_vg^-$, a contradiction.

We also claim that no new annoying endpoints are added.  In the first case, $u_4$ cannot be, since it has only one neighbors on $P_2$, and if $x$ is one, then $u_2$ must be adjacent to a vertex in $u_5Py$, which is impossible.  In the second case, no vertex is an annoying endpoint as $P_2$ is a $2$-path. In the last case, $x$ cannot be as it has only one neighbor on $P_1$, and if $u_2$ is one, then $u_4$ should be light and be adjacent to a PE-vertex, but $u_4x_v\in E(G)$ and $x_v$ is not a PE-vertex in $P_2$.

(c2).  $P_w\not=P_v$, and $v^-, v^+$ are adjacent to $x_w, x$ or $y$.   If $xv^-\in E(G)$, then $P, P_v, P_w$ can be combined into two paths: $P_wx_wu_4Py$ and $x_vP_vv^-xu_3u_2u_1vP_vy_v$, a contradiction. So we may assume that $xv^-, xv^+\not\in E(G)$.  It follows, by symmetry, that $yv^+\in E(G)$. But again, $P, P_v, P_w$ can be combined into two paths:  $P_wx_wu_4Pyv^+P_vy_v$ and $x_vP_vvu_1u_2u_3x$, a contradiction.
\end{proof}

\section{Weights on paths}\label{counting}

We give an initial weight of $10$ to each path in $\P$.  By Lemma~\ref{no-triangle}, all paths in $\P$ are non-cyclic paths with order more than $1$. So we may think that each endpoint of the paths in $\P$ gets an initial weight of $5$. Here is the rule to transfer weights between (vertices on) paths:
\begin{quote}
{\bf Rule to transfer weights:} Each endpoint sends a weight of $2$ to the adjacent weighty or heavy vertex, and each PE-vertex transfers $1$ to the adjacent light vertex.
\end{quote}

For convenience, we let $w(P)$ be the final weight on a path or a segment $P$. For a path $P\in \P$, let $s_1(P), s_2(P), s_3(P)$ and $n_o(P)$ be the number of weighty and heavy vertices,  light vertices, neutral vertices, and PE-vertices, respectively.  Then 
 \begin{equation}\label{weight}
  w(P)=2+2s_1(P)+s_2(P)-n_o(P).
  \end{equation}

\begin{definition}
Let $P=xu_1u_2\ldots u_ky\in \P$.  For each $1\le i\le j\le k-1$,
\begin{itemize}
\item a neutral vertex $u_i$ is {\em free} if there are neither heavy nor weighty vertices on $P$ between $u_i$ and an endpoint ($x$ or $y$);
\item $u_iPu_j$ is a {\em heavy segment} if both $u_i$ and $u_j$ are heavy or weighty and there is no neutral vertices on it;  (so a single heavy or weighty vertex is also a heavy segment)
\item $u_iPu_j$ is a {\em neutral segment} if $u_i, u_j$ are non-free neutral and there is no heavy or weighty vertices on it. (so a single neutral vertex is also a neutral segment)
\end{itemize}
\end{definition}

Note that a light vertex may be on a heavy or neutral segment.  By Corollary~\ref{neutral} and Lemma~\ref{no-consecutive}, there are at least one neutral vertices between any two heavy vertices, so a heavy segment with more than one vertices must contain at least one weighty vertex,  and contain at most three vertices, and when it contains three vertices, it must be a light vertex adjacent to two weighty vertices.  Let a {\em heavy pair} be a pair of vertices in a heavy segment that are both heavy or weighty.  So every heavy segment contains either $0$ or $1$ heavy pair.  Let $n_h(P)$ be the number of heavy pairs on $P\in \P$.  So $n_h(P)\le 4$ for each $P\in \P$. 

Similarly, a pair of neutral vertices in a neutral segment that are consecutive or separated by a light vertex is called a {\em neutral pair}.  So a neutral segment with $s$ neutral vertices contains $s-1$ {\em neutral pairs}.  A heavy segment (and similarly, a neutral segment) is {\em maximal} if it is not contained in a larger heavy segment (neutral segment).

\begin{lem}\label{weight-path}
For a path $P\in \P$ with $w(P)>|V(P)|$, $$n_h(P)\ge n_o(P)+n_q(P)+n_r(P),$$ where $n_h(P), n_o(P), n_q(P), n_r(P)$ are the numbers of heavy pairs, PE-vertices, neutral pairs, and free neutral vertices on $P$, respectively.
\end{lem}

\begin{proof}
As above, assume that $P$ contains $a$ maximal heavy segments, then $P$ contains $a-1$ maximal neutral segments. It follows that $s_1(P)=n_h(P)+a$ and $s_3(P)=n_q(P)+a-1+n_r(P)$. By \eqref{weight}, 
\begin{align*}
w(P)&=2+s_1(P)+s_2(P)+(n_h(P)+s_3(P)-n_q(P)+1-n_r(P))-n_o(P)\\
&=|V(P)|+n_h(P)+1-(n_q(P)+n_r(P)+n_o(P)).
\end{align*}
It follows that if $w(P)>|V(P)|$ then $n_h(P)\ge n_q(P)+n_r(P)+n_o(P)$, as claimed. 
\end{proof}

For convenience, we call the number $n_o(P)+n_q(P)+n_r(P)$ the {\em good number of $P$}, and in particular, it is called the {\em good number of the segment} if $P$ is a segment of a path.

\begin{lem}\label{ClaimA}
If for some $i>1$, $xu_i\in E(G)$ and $u_{i+1}$ is heavy (so $\{u_i, u_{i+1}\}$ is a heavy pair), then
 the good number of $xPu_i$ is at least $1$, with equality if and only if $i=4$ and $u_2$ is heavy.
 \end{lem}

\begin{proof}  First of all,  $xPu_{i-1}$ contains no neighbors of $y$ by Corollary~\ref{no-neighbor-of-y}, and   contains at most one heavy or light vertex by Lemma~\ref{paired-cycle}.  We may assume that the good number of $xPu_i$ is at most $1$.

Clearly, $i>2$, or $xu_1u_2$ is a net.  Also, $i\not=3$, or by Lemma~\ref{no-c}, both $u_1, u_2$ are neutral (and free), so the good number of $xPu_i$ is $2$.    So $i\ge 4$, and $xPu_{i-1}$ contains at least one heavy, weighty, or light vertices, or there are at least two free neutral vertices on $xPu_i$. 

First assume that $xPu_{i-1}$ contains no weighty vertices. Then it must contain exactly one heavy or light vertex, by Lemma~\ref{paired-cycle}.  If it contains a light vertex, then all neutral vertices on $xPu_{i-1}$ are free, thus the good number is at least $i-1-1\ge 2$.  Therefore, it contains a heavy vertex.  If $u_1$ is heavy, then $u_2, \ldots, u_{i-1}$ are all neutral, so it contains $i-3$ distinct neutral pairs, namely, $\{u_2, u_3\}, \{u_3, u_4\}, \ldots, \{u_{i-2}, u_{i-1}\}$;  If $u_2$ is heavy, then $u_1$ is free and $\{u_3, u_4\}, \ldots, \{u_{i-2}, u_{i-1}\}$ are $i-4$ distinct neutral pairs; if $u_1, u_2$ are not heavy, then $u_1, u_2$ are free neutral vertices.   So in either case, the good number is at least $i-3$.  Then $i=4$ and $u_1$ or $u_2$ is heavy.  By Lemma~\ref{no-7}, $u_1$ cannot be heavy, so $u_2$ is heavy.

Now assume that $xPu_{i-1}$ contains a weighty vertex, say $u_j$ for some $1<j<i$.  Then $xu_j\in E(G)$. If $j=2$, then $u_1u_3\in E(G)$, or we will have a net.  Note that $u_4\not=u_i$ (otherwise $u_4u_5$ is a cut-edge),  $u_4$ cannot be neutral (otherwise, $\{u_3,u_4\}$ is a neutral pair and $u_1$ is free, so the good number of $xPu_i$ is at least $2$), and $u_4$ cannot be light  (otherwise, $u_5$ must be neutral by Corollary~\ref{neutral-end}, thus $\{u_3, u_5\}$ is a neutral pair and $u_1$ is free, so the good number of $xPu_i$ is at least $2$).  So $u_4$ is heavy.  Now $u_5$ is neutral and $u_6$ must be $u_i$, or $\{u_5, u_6\}$ is a neutral pair.  But then by Lemma~\ref{no-two-consecutive}, $u_5$ must be adjacent to some vertex on $xPu_i$ other than $u_4,u_6$, which is impossible.

So  $j\ge 3$,  and one vertex on $xPu_{j-1}$ must be heavy or light (otherwise there are at least two free neutral vertices).   As $xPu_{j-1}$ contains no weighty vertices,  the above argument shows that $xPu_{j-1}$ has good number at least $2$, unless $j=4$ and $u_2$ is heavy.  In the bad case, $i=j+2=6$ (for otherwise $\{u_{j+1}, u_{j+2}\}$ is a neutral pair), and by Lemma~\ref{no-two-consecutive}, $u_1, u_3, u_5$ must be adjacent only to vertices on $xPu_6$, which is impossible.  
\end{proof}

\begin{lem}\label{heavy-triple}
Let $P=xu_1\ldots u_ky\in \P$. If for some $1<i<k$, $xu_i, yu_{i+2}\in E(G)$ and $u_{i+1}$ is light, then $w(P)\le |V(P)|$.
\end{lem}

\begin{proof}
Assume that $w(P)>|V(P)|$. Then by Lemma~\ref{weight-path}, $n_h(P)\ge n_o(P)+n_q(P)+n_r(P)$. 

By Lemma~\ref{paired-cycle}, $xPu_i$ and $yPu_{i+2}$ contain at most one heavy vertex altogether.  By Corollary~\ref{no-neighbor-of-y}, $xPu_i$ contains no neighbors of $y$ and $u_{i+1}Py$ contains no neighbor of $x$. We may assume that $yPu_{i+2}$ contains no heavy vertices.   As $i+2\not=k-1$ (otherwise $u_{i+2}u_ky$ is a net),  $yPu_{i+2}$ contains at least one free neutral vertex, so the good number of $yPu_{i+2}$ is at least one.

If $xPu_i$ contains no heavy vertices, then similarly the good number of $xPu_i$ is also at least one, so the good number of $P$ is at least $2$,  but $P$ has only one heavy pair, a contradiction.  Thus,  we may assume that $xPu_i$ contains exactly one heavy vertex.

If $xPu_i$ contain no heavy pairs, then its good number must be zero. It follows that $u_1$ is heavy (otherwise $n_r\ge 1$), $xu_3\in E(G)$ (otherwise $n_q\ge 1$), and $i\ge 3$.  As $u_2$ cannot be a PE-vertex (otherwise $n_o+n_r+n_q\ge 2$ but $n_h=1$), $u_2u_a\in E(G)$ for some $a\not=1,3$. By Lemma~\ref{no-two-consecutive}, $i>a$.    By Lemma~\ref{chord} (b), $u_{a-1}$ is neutral, so $n_q\ge 1$ thus $n_q+n_r\ge 2$, a contradiction.

Thus,  we may assume that $xPu_i$ contains a heavy pair, that is, $xu_j\in E(G)$ and $u_{j+1}$ is heavy for some $2\le j<i-1$. By Lemma~\ref{ClaimA}, the good number of $xPu_j$ is at least $1$.  Thus it must be exactly $1$, $j=4$ and $u_2$ is heavy. But then $xPu_i$ contains two heavy vertices, a contradiction. 
\end{proof}

\begin{lem}\label{main-lemma}
For each $P\in \P$, $w(P)\le |V(P)|$.
\end{lem}

\begin{proof}
Let $P=xu_1u_2\ldots u_ky\in \P$. By Lemma~\ref{no-7}, $P$ is non-cyclic and is not a $1$-path or a $3$-path.  Assume that $w(P)>|V(P)|$.  Let $n_h, n_o, n_q, n_r$, as defined in the Lemma~\ref{weight-path},  be the numbers of heavy pairs,  PE-vertices, neutral pairs, and free neutral vertices on $P$, respectively. By Lemma~\ref{weight-path}, 
\begin{equation}
n_o+n_q+n_r\le n_h\le 4.
\end{equation}

By Lemma~\ref{heavy-triple} and Lemma~\ref{no-consecutive}, we may assume that each heavy pair consists of two consecutive vertices on $P$, at least one of which is a weighty vertex. 

\smallskip

{\bf Case 1.} $P$ contains a heavy pair consisting of two weighty vertices. Assume that $xu_i, yu_{i+1}\in E(G)$ for some $1<i<k-1$.  Note that by Lemma~\ref{no-two-consecutive}, $xu_{i-1}, yu_{i+2}\not\in E(G)$.

{\bf Case 1.1.} $P$ contains at least two heavy pairs.

First assume that $P$ contains another heavy pair consisting of two weighty vertices, say $\{u_j, u_{j+1}\}$ with $xu_j, yu_{j+1}\in E(G)$ for some $j>i$.  In this case, there are no other heavy pairs.  So $n_h=2$.   As $P$ can be rerouted so that $u_1$ and $u_k$ are the endpoints,   $u_1, u_k$ must be neutral by Corollary~\ref{neutral-end}, thus free, so $n_r=2$ and $n_o=n_q=0$.   Then $u_{j-1}$ is not in a neutral pair. It follows that $u_{j-2}$ is heavy or $u_{j-2}=u_{i+1}$, and in either case, $u_{j-1}u_s\in E(G)$ for some $s\not=j-2, j$, or $u_{j-1}$ is a PE-vertex (a contradiction to $n_o=0$).  By Lemma~\ref{chord}, $u_{s+1}$ (if $s<j$) or $u_{s-1}$ (if $s>j$) is neutral, so we have a neutral pair, a contradiction to $n_q=0$.

Now we assume that $\{u_i, u_{i+1}\}$ is the only heavy pair containing two weighty vertices. Note that if a heavy pair is $\{u_s, u_{s+1}\}$ such that $xu_s\in E(G)$ and $u_{s+1}$ is heavy, then by Corollary~\ref{no-neighbor-of-y}, $s<i$.   Similarly for such a heavy pair involving $y$.  Also note that $n_h\le 3$. 

Consider the case $n_h=3$.    Let $\{u_s, u_{s+1}\}$ and $\{u_t, u_{t-1}\}$ be the other two heavy pairs such that $s<i, t>i+1$, $xu_s, yu_t\in E(G)$, and $u_{s+1}, u_{t-1}$ are heavy.  By Lemma~\ref{paired-cycle}, $P$ contains no other heavy or light vertices.  Now by Lemma~\ref{ClaimA},  each of $xPu_i$ and $u_{t+1}Py$ has the good number at least two, thus the good number of $P$ is at least $4$, a contradiction.

Now let $n_h=2$ and $\{u_s, u_{s+1}\}$ be the only other heavy pair such that $xu_s\in E(G)$ and $u_{s+1}$ is heavy.  By Lemma~\ref{ClaimA}, $xPu_s$ has the good number at least $2$ or $s=4$ such that $u_2$ is heavy.  In the latter case, $u_{i+1}Py$ contains no heavy or light vertices by Lemma~\ref{paired-cycle}, so has at least two free neutral vertices, thus the good number of $P$ is at least $3$, a contradiction.  For the former case, the good number of $u_{i+1}Py$ is $0$. It follows that $u_k$ must be heavy, and by Lemma~\ref{paired-cycle}, there are no other heavy or light vertices on $P$ other than $u_{s+1}$ and $u_k$.  Then $s=3$ and $i=6$.  Now $u_{i-1}$ must be adjacent $u_t$ for some $t\not=i-2, i$.  Clearly, $t\not=1,2, i+2$.  Then $t>i+1$, and by Corollary~\ref{neutral},  $u_{t-1}$ is neutral, so $\{u_{t-1}, u_t\}$ is a neutral pair, a contradiction to the fact that the good number of $u_{i+1}Py$ is $0$.

{\bf Case 1.2.} $\{u_i, u_{i+1}\}$ is the only heavy pair on $P$.  Then the good number of $P$ is at most one.  By symmetry, we may assume that the good number of $xPu_i$ is at most one and the good number of $u_{i+1}Py$ is $0$.  It follows that $u_k$ cannot be neutral (otherwise it is free).

Let $xu_j\in E(G)$ for some $j>i$.  Then $P$ can be rerouted so that $u_1$ is an endpoint.  By Lemma~\ref{neutral-end}, $u_1$ is neutral (and free).  So $n_r\ge 1$. It follows that $n_r=1$ and $n_o=n_q=0$.  By Lemma~\ref{no-two-consecutive}, $u_{j-1}u_s\in E(G)$ for some $s\not=j-2, j$.  By Corollary~\ref{neutral}, $u_{s+1}$ (if $s<j$) or $u_{s-1}$ (if $s>j$) is neutral, thus $n_q\ge 1$, a contradiction. It follows that $u_{i+1}Py$ contains no neighbor of $x$, and by symmetry, $xPu_i$ contains no neighbour of $y$. 

Now that $u_k$ cannot be neutral or weighty, it must be light or heavy.  But if it is light, then $u_{k-1}$ cannot be a neighbour of $y$ (otherwise $yu_ku_{k-1}$ is a net), or heavy by Lemma~\ref{no-consecutive}, so $u_{k-1}$ is neutral, thus the good number of $u_{i+1}Py$ is not $0$, a contradiction.  So we assume that $u_k$ is heavy.

Observe that $u_{k-1}$ is neutral.   We may assume that $u_{k-2}$ is not neutral, or we have a neutral pair $\{u_{k-2}, u_{k-1}\}$ on $u_{i+1}Py$, a contradiction.  So $u_{k-2}$ is heavy, light or $yu_{k-2}\in E(G)$. 

Note that $u_{k-2}$ cannot be heavy.  Suppose otherwise. By Lemma~\ref{paired-cycle}, $xPu_i$ contains no heavy or light vertices.   So $u_1, u_2$ are free neutral vertices. It follows that $n_r\ge 2$, a contradiction.

We also have $u_{k-2}y\not\in E(G)$. Suppose otherwise. Then $u_{k-1}u_s\in E(G)$ for some $i+1<s<k-2$, or $u_{k-1}$ is a PE-vertex.    By Lemma~\ref{chord}, $u_{s+1}$ is neutral, thus $\{u_s, u_{s+1}\}$ is a neutral pair on $u_{i+1}Py$, a contradiction.

So let $u_{k-2}$ be light.  By Lemma~\ref{no-consecutive},  $u_{k-3}$ is not heavy or light.  But $u_{k-3}$ is not a neighbour of $x$ or $y$, so $u_{k-3}$ must be neutral.  Now $\{u_{k-3}, u_{k-1}\}$ is a neutral pair, a contradiction to the assumption that the good number of $u_{i+1}Py$ is $0$.

\smallskip

{\bf Case 2.}  $P$ contains no heavy pairs consisting of only weighty vertices.

{\bf Case 2.1.} $n_h=0$. It follows that $n_o=n_q=n_r=0$.  Recall that $k>1$, by Lemma~\ref{no-triangle}. 

Assume first that $P$ contains no weighty vertices.  Then $u_1, u_k$ are heavy, and the vertices on $u_1Pu_k$ are alternatively heavy and neutral.  Let $v\not=u_1, u_3$ be the third neighbor of $u_2$. Then $v=u_t\in P$ for some $t>3$, or $u_2$ is a PE-vertex (this contradicts to the fact that $n_o=0$).  But $u_{t-1}$ and $u_{t+1}$ are heavy. So  by Lemma~\ref{chord} (a), $u_1,u_{t-1}$ are adjacent to one endpoint of a path $Q\not=P$ and $u_3, u_{t+1}$ are adjacent to the other endpoint of $Q$.  Note that $t\not=4$, or $u_2u_3u_4$ is a net.   Now $u_4$ is neutral and must be adjacent to $v'\not\in P$ or $u_s$ for some $s>5$. In the former case, $u_{i+3}$ is a PE-vertex, and in the latter case,  $u_{s-1}$ or $u_{s+1}$  cannot be heavy by Lemma~\ref{chord} (a), a contradiction.

Now let $xu_i\in E(G)$ for some $i>1$, and we may assume that $xu_j\not\in E(G)$ for $1<j<i$.  First let $i>2$.   Then  $u_{i-2}$ cannot be neutral (or $\{u_{i-2}, u_{i-1}\}$ is a neutral pair) or light (otherwise, $u_{i-1}$ is free if $i=3$, or $u_{i-3}$ must be neutral thus $\{u_{i-3}, u_{i-1}\}$ is a neutral pair).  So $u_{i-2}$ is heavy or $yu_{i-2}\in E(G)$.  Now $u_{i-1}$ must be adjacent to $u_t$ for some $t\not=i-3,i-1$, or it is a PE-vertex.  By Lemma~\ref{chord} (b), $u_{t-1}$ (if $t>i-1$) or $u_{t+1}$ (if $t<i-3$) is neutral, thus $u_t$ is in a neutral pair, a contradiction to $n_q=0$.  Now let $i=2$.  Then $u_1u_3\in E(G)$, or $xu_1u_2$ is a net.  So $xu_j\in E(G)$ for some $j>3$, or $u_1$ is a better endpoint of $P$ than $x$.  Clearly $j\not=4$, or $u_4u_5$ is a cut-edge.  So $u_{j-2}$ must be heavy or adjacent to $y$. Now a similar argument as above shows $n_q>0$, a contradiction.

{\bf Case 2.2.}  $n_h=1$.  Let $\{u_i, u_{i+1}\}$ be the only heavy pair with $xu_i\in E(G)$ for some $1<i<k$  and $u_{i+1}$ is heavy.   As the good number of $xPu_i$ is at most $1$,  by Lemma~\ref{ClaimA},  $i=4$ and $u_2$ is  heavy.  So $u_1$ is free, and $n_q=n_o=0$.  

Now,  $u_3u_t$ for some $t\not=2,4$, or $u_3$ is a PE-vertex, a contradiction to $n_o=0$.  Clearly, $t\not=1$, so $t\ge 6$. By Lemma~\ref{neutral-end}, $u_{t-1}$ is neutral or $xu_{t-1}\in E(G)$. In the former case, we have a neutral pair, thus $n_q>0$, a contradiction. As $y$ has no neighbors on $xu_1u_2u_3$, Lemma~\ref{no-two-consecutive} implies that the latter case cannot happen.

{\bf Case 2.3.}  $n_h=2$.  Let $\{u_i, u_{i+1}\}$ and $\{u_{t-1}, u_t\}$ be the two heavy pairs for some $i<t-1$.

First assume that $xu_i, xu_{t-1}\in E(G)$ and $u_{i+1}, u_t$ are both heavy.  By Lemma~\ref{paired-cycle}, $u_{i+1},u_t$ are adjacent to the same endpoint $x'\in P'\in \P$, and they are the only heavy/light vertices on $xPu_t$. As $i\ge 3$ (otherwise there is a net) and $u_1, u_2, \ldots, u_{i-1}$ are free neutral vertices, $n_r\ge 2$. It follows that $n_r=2$ and $i=3$, and $n_o=n_q=0$.  By Lemma~\ref{no-two-consecutive}, $u_2$ can only be adjacent to vertices on $P$, thus $u_2u_a\in E(G)$ for some $a\ge 4$, and by Lemma~\ref{chord}, $u_{a-1}$ is neutral, so $n_q\ge 1$, a contradiction.

Note that if $xu_{t-1}, yu_{i+1}\in E(G)$ and $u_t, u_i$ are heavy, then $P$ can be rerouted so that $u_i$ and $u_t$ are endpoints, a contradiction to Lemma~\ref{neutral-end}.  Thus, we assume that $1<i<t<k$, $xu_i, yu_t\in E(G)$ and $u_{i+1}, u_{t-1}$ are heavy, and $y$ has no neighbors on $xPu_i$ and $x$ has no neighbors on $u_tPy$.  Let $u_{i+1}$ be adjacent to an endpoint $x'\in P'\in \P-\{P\}$.

By Lemma~\ref{ClaimA},  $xPu_{i-1}$ and $u_{t+1}Py$ both have the good number at least $1$, thus $i=4, t=k-3$ and $u_2$ and $u_{k-1}$ are heavy.  So there are two free neutral vertices, and $n_q=n_o=0$.   Note that $i+1<t-1$.  For otherwise, by Lemma~\ref{paired-cycle},  $xPu_i$ and $u_tPy$ contain at most one heavy/light vertex altogether, a contradiction.

Now, $u_3u_t$ for some $t\not=2,4$, or $u_3$ is a PE-vertex.  Clearly, $t\not=1$, so $t\ge 6$. By Lemma~\ref{neutral-end}, $u_{t-1}$ is neutral or $xu_{t-1}\in E(G)$. In the former case, we have a neutral pair, a contradiction. As $y$ has no neighbors on $xu_1u_2u_3$, Lemma~\ref{no-two-consecutive} implies that the latter case cannot happen.

{\bf Case 2.4.}  $n_h\in \{3,4\}$.  As in Case 2.3,   we may assume that the neighbors of $x$ are $u_1, u_i, u_j$ and the neighbor of $y$ are $u_s, u_t, u_k$ with $1<i<j<s\le t<k$ such that $u_{i+1}, u_{j+1}, u_{s-1}, u_{t-1}$ are all heavy.

By Lemma~\ref{paired-cycle}, $u_{i+1}, u_{j+1}$ are adjacent to the same endpoint of a path in $\P$, and there are no other heavy or light vertices on $xPu_j$. Note that $i\ge 3$ (otherwise $xu_1u_2$ is a net), so $xPu_i$ contains at least $i-1\ge 2$ free neutral vertices, i.e., $u_1, u_2,\ldots, u_{i-1}$.  Similarly, $u_sPy$ contains at least $k-t\ge 2$ free neutral vertices when $n_h=4$ (that is, $s<t$). Note that when $n_h=3$ (that is, $s=t$), the good number of $u_sPy$ is $1$, thus by Lemma~\ref{ClaimA}, $s=k-3$, $u_{k-1}$ is heavy and $u_k$ is free.  So $P$ contains $n_h$ free neutral vertices. It follows that $n_q=n_o=0$, and $i=3$ and $j=6$.    By Lemma~\ref{no-two-consecutive},  $u_5$ must be adjacent to $u_1$ or $u_2$,  but then $P, P'$ can be combined into one path: $P'x'u_4u_5u_1u_2u_3xu_6Py$ if $u_5u_1\in E(G)$, and $P'x'u_4u_3xu_1u_2u_5Py$ if $u_5u_2\in E(G)$, a contradiction.
\end{proof}

\begin{proof}[Proof of Theorem~\ref{main}:]   
Consider an optimal path cover $\P$ of $G$, and assign a weight of $10$ to each path in $\P$.  By Lemma~\ref{main-lemma}, the total weight is $10|\P|=\sum_{P\in \P} w(P)\le \sum_{P\in \P} |V(P)|=n$.  So $|\P|\le n/10$, that is, $G$ has a path cover with at most $n/10$ paths. 
\end{proof}

{\bf Acknowledgement:} The author would like to thank the referees for their valuable comments, which helped to simplify the proof and improve the presentation of the paper significantly.

\end{document}